\documentclass{article}
\usepackage{graphicx} 
\usepackage{comment}
\usepackage{amssymb}
\usepackage{amsthm}
\usepackage{extarrows}
\usepackage{calrsfs}
\usepackage{tikz-cd}
\usepackage{mathrsfs}
\usepackage{dynkin-diagrams}
\usepackage{array}
\usepackage[colorlinks = true,citecolor=blue,linkcolor=blue]{hyperref}

\newtheorem{definition}{Definition}[section]
\newtheorem{theorem}[definition]{Theorem}
\newtheorem{lemma}[definition]{Lemma}

\newtheorem{proposition}[definition]{Proposition}

\newtheorem{corollary}[definition]{Corollary}

\newtheorem{conjecture}[definition]{Conjecture}

\newtheorem{remark}[definition]{Remark}

\newtheorem{example}[definition]{Example}

\usepackage{booktabs}
\usepackage{array}
\usepackage{adjustbox}
\usepackage{dynkin-diagrams} 
\newcolumntype{L}{>{\raggedright\arraybackslash}p{0.45\linewidth}}
\newcolumntype{C}{>{\centering\arraybackslash}p{0.45\linewidth}}

\title{Curvature positivity for Kähler and quasi-Kähler flag manifolds}

\author{Giovane Galindo \ \  Ailton R. Oliveira}
\date{}

\begin{document}
\maketitle
\begin{abstract}
    In this paper, we study the notions of Griffiths and dual-Nakano positivity for the curvature of the Chern connection on Kähler and quasi-Kähler flag manifolds, as well as for the complex projective space. In this setting, we prove that every flag manifold endowed with a complex structure admits a metric of dual-Nakano semi-positive curvature, and we give a full classification of Kähler flag manifolds with Griffiths semi-positive curvature. Next we prove a series of restrictions for a quasi-Kähler flag manifold to have Griffiths semi-positive curvature, and we conjecture that in fact, there are no such metrics for non-integrable almost-complex structures. Lastly, we give a full classification on invariant metrics on the complex projective space with Griffiths and dual-Nakano semi-positive curvature.
\end{abstract}

\section{Introduction}

The curvature positivity of a Hermitian vector bundle $(E,h)$ was first introduced by Nakano in \cite{nakano1955complex}, this notion of positivity is known as Nakano positivity. Later, Griffiths introduced the notion of Griffiths positivity \cite{griffiths1969hermitian}. These notions of curvature positivity on Hermitian vector bundles and related topics have been extensively studied in works such as \cite{67d6e687-a6fb-3db7-9e87-e61bb1083959}, \cite{e936a794-9243-34c8-996f-d7f8f1243d24} \cite{Bloch1971} and \cite{green2021positivity}.\\

Let \(M\) be a complex manifold endowed with a Hermitian metric \(g\).  One may then study various notions of curvature positivity for the holomorphic tangent bundle \((T^{1,0}M,\,g)\).  In the Kähler setting, these notions coincide with positivity of the holomorphic bisectional curvature \cite{siu1980compact,mok1988uniformization}.  However, on a general (non-Kähler) Hermitian manifold, the complex structure is not compatible with the Levi–Civita connection.  It is therefore customary to replace it with the Chern connection, which is compatible with both the metric and the complex structure, though it generally has nonzero torsion.\\

The Griffiths semi-positivity is closely related to Donaldson's conjecture \cite{donaldson2006two}. Let $(M,J,g)$ be a 2n dimensional almost-Kähler compact manifold, with fundamental $2$-form $\omega$, the conjecture states:
\begin{conjecture}
    Let $\sigma$ be a smooth volume form on $(M,J,\omega)$, with $\int_M \sigma = \int_M \omega^n$, if $ \overline{\omega}$ is an almost-Kähler metric on $(M,J)$ with $\overline{\omega}^n = \sigma$ and $[\overline{\omega}] = [\omega]$, that is $\overline{\omega}-\omega$ is a closed 2-form. Then, there are $C^{\infty}$ a priori bounds on $\overline{\omega}$ depending only on $\omega, J$ and $\sigma$.
\end{conjecture}

And the following result by \cite{tosatti2008taming} relates the conjecture to the positivity of a slight alteration of the curvature tensor.

\begin{theorem}
   Let $(M,J,\omega)$ be an almost-Kähler manifold, define a tensor
\[
\mathcal{R}_{i\overline{j}k\overline{l}}(g,J) = R_{k\overline{l}i}^{j} + 4\, N^r_{\overline{l}\overline{j}}\, \overline{N_{\overline{r}\overline{k}}^i},
\]
where \(R_{k\overline{l}i}^{j}\) is the \((1,1)\) part of the curvature of the Chern connection and \(N\) represents the Nijenhuis tensor. 
If the tensor $\mathcal{R}$ is semi-positive in the Griffiths sense, i.e.,
\[
\mathcal{R}(X,\overline{X}, Y,\overline{Y)}\geq 0 ,
\]
for all \((1,0)\) vectors \(X\) and \(Y\), then $(M,J,\omega)$ satisfies Donaldson's conjecture.
\end{theorem}

Another important feature of Griffiths-positivity on non-Kähler complex manifold is its relation to one Hermitian curvature flow \cite{streets2011hermitian}. Ustinovskiy defined the positive Hermitian curvature flow in \cite{ustinovskiy2019hermitian}, for Hermitian manifolds, by
\begin{equation}
    \begin{cases}
        &\frac{\partial g}{\partial t} = -S-Q\\
        &g(o)=g_0 ,
    \end{cases}
\end{equation}
where $S$ is the Ricci curvature of the Chern connection and $Q$ is a certain quadratic polynomial in the torsion of the underlying Chern connection. He proved the short-time existence of the flow and showed that, if the initial metric is Griffiths positive (semi-positive), then the solution to the flow is always Griffiths positive (semi-positive). Later on, he applied this theory to homogeneous spaces \cite{ustinovskiy2019hermitian}, \cite{ustinovskiy2021lie}.

In this paper, we want to study the curvature positivity of flag manifolds, they are a special class of homogeneous spaces, in which one can use results from semisimple Lie theory to study geometric properties of the manifold, see \cite{grama2023scalar}, \cite{correa2024twisted}, \cite{grama2022projected}, \cite{grama2021scalar}, \cite{san2003invariant}, \cite{correa2023t}.

Firstly, we use a result from \cite{ustinovskiy2019hermitian} to show that every flag manifold admits a metric of dual-Nakano semi-positive curvature.
\begin{proposition}
    Let $M=G/K$ be a flag manifold equipped with an integrable complex structure $J$, let
\begin{equation*}
    \mathfrak{g} = \mathfrak{k} \oplus \mathfrak{m}
\end{equation*}
be the reductive decomposition, then the metric $-B|_{\mathfrak{m}}$ on $(M,J)$ is dual-Nakano semi-positive.\\
In other words, on every flag manifold with a complex structure, the metric $\lambda$ given by $\lambda_\alpha = 1$, for all $\alpha \in R_M$ is a dual-Nakano semi-positive metric.
\end{proposition}

By combining the above result with Mok’s classification of Kähler manifolds that are Griffiths semi-positive \cite{mok1988uniformization}, we obtain a complete classification of Griffiths semi-positive Kähler flag manifolds.

\begin{theorem}
    Let $(X,g)$ be a compact, simply connected Kähler manifold with Griffiths semi-positive curvature, if $g$ is invariant by the action of a transitive group then, X is isometrically biholomorphic to
\begin{equation*}
    (M_1 , g_1 ) \times ... \times (M_p , g_p ),
\end{equation*}
where $M_i$ are the Hermitian Symmetric spaces and $g_i$ its unique invariant metric.
\end{theorem}

Next, we allow the definitions of curvature positivity to include almost-Hermitian manifolds. And, by using the theory of roots on semisimple Lie algebras, we compute the curvature of the Chern connection on flag manifolds. By requiring the metric to be quasi-Kähler, we use these computations to obtain the following results.

\begin{theorem}
    Let $M=G/K$ be a flag manifold, $J$ a non-integrable almost-complex structure and $\omega$ a quasi-Kähler metric. Suppose $(M,J,\omega)$ has Griffiths semi-positive curvature, then $G$ has a simple root of mark $3$.
\end{theorem}

\begin{corollary}\label{cor1} 
None flag manifold associated with $A_l , B_l , C_l$ and $D_l$ accepts a quasi-Kähler metric on a non-integrable almost-complex structure that is Griffiths semi-positive.    
\end{corollary}

\begin{theorem}
    Let $M$ be a maximal flag manifold. There are no almost-complex structure $J$ and quasi-Kähler metric $\omega$ such that $(M,J,\omega)$ has Griffiths semi-positive curvature.
\end{theorem}

\begin{proposition}
Let $M$ be a flag manifold associated with $G_2$ and $J$ a non-integrable almost-complex structure. There is no quasi-Kähler metric $\lambda$, such that $(M , J ,\lambda)$ is Griffiths semi-positive.
\end{proposition}

Since none of the classical simple Lie algebras of types \(A_l\), \(B_l\), \(C_l\), or \(D_l\) possess a root of mark~3, follows from the Corollary \ref{cor1} the corresponding flag manifolds cannot support a quasi-Kähler metric with Griffiths semi-positive curvature on a non-integrable almost-complex structure.  Likewise, the exceptional Lie algebra \(G_2\) and all maximal flag manifolds fail to admit such metrics.  These observations lead us to the following conjecture.

\begin{conjecture}
If a flag manifold admits an almost-complex structure $J$ and a quasi-Kähler metric $\lambda$ of Griffiths semi-positive curvature then, it is isometrically biholomorphic to
\begin{equation*}
   (M_1 , g_1 ) \times ... \times (M_p , g_p ),
\end{equation*}
where $M_i$ are the Hermitian Symmetric spaces and $g_i$ its unique invariant metric.
\end{conjecture}

Lastly, the following result from \cite{ustinovskiy2019hermitian} states that any manifold with Griffiths positive curvature is biholomorphic to the complex projective space.

\begin{theorem}
    Let $(M,J,g )$ be a compact Hermitian manifold such that
    \begin{enumerate}
        \item [a)] the Chern curvature is Griffiths semi-positive,
        \item[b)] the Chern curvature is Griffiths positive in one point.
    \end{enumerate}
    Then, $M$ is biholomorphic to the complex projective space $\mathbb{C}P^n$.
\end{theorem}

We then classify all invariant Hermitian metrics on  $\mathbb{C}P^{2n-1} =Sp(n)/Sp(n-1)\times U(1)$ that have this property, completing the characterization of positive invariant metrics on the complex projective space. More precisely, we prove the following.

\begin{theorem}
The projective space $\mathbb{C}P^{2n-1} =Sp(n)/Sp(n-1)\times U(1)$ endowed with the integrable almost-complex structure $(+,+)$ and the invariant metric $(1,t)$ is Griffiths and dual-Nakano semi-positive if, and only if, $t\geq 1$. Moreover, $\mathbb{C}P^{2n-1}$ is Griffiths and dual-Nakano positive if, and only if, $t> 1$.
\end{theorem}

\section{Almost-complex geometry}\label{Almost-complex}
In this section, we give a brief overview on almost-complex geometry. Focusing on some special types of metrics and on the notions of curvature positivity. We start by recalling some basic definitions for almost-complex geometry.

\begin{definition}
    Let $M$ be a smooth manifold, an almost-complex structure $J$ on $M$ is a map $J: TM \rightarrow TM$ that satisfies $J^2 = -Id$. If $g$ is a Riemannian metric on $M$ and $J$ is $g$-orthogonal, we call $(M,g,J)$ an almost-Hermitian manifold.
\end{definition}

\begin{remark}
    Note that, if $J$ is an almost-complex structure, so is $-J$, this is called the conjugate of $J$ and we will consider them as the same structure.
\end{remark}

On almost-Hermitian manifolds, we define the fundamental 2-form as follows,
\begin{equation}
    \omega(X,Y) = g(JX,Y).
\end{equation}
We may also refer to $\omega$ as the metric of the almost Hermitian manifold.\\

\begin{definition}
  An almost-complex structure $J: TM \rightarrow TM$ is called integrable (or a complex structure) if, the Nijenhuis tensor 
\begin{equation}
    N(X,Y) = [JX,JY]-J[JX,Y]-J[X,JY]-[X,Y]
\end{equation}
vanishes identically. In that case, we call $(M,g,J)$ a Hermitian manifold. 
\end{definition}

Next, we define the special types of metrics on almost-Hermitian manifolds, which we will be interested in.

\begin{definition}
    Let $(M,\omega,J)$ be an almost-Hermitian manifold. We call it
    \begin{enumerate}
        \item [1)] a Kähler manifold if $J$ is integrable and $d\omega=0$,
        \item[2)] a quasi-Kähler manifold (or $(1,2)$-symplectic) if $d\omega^{1,2}=0$,
        \item [3)] an almost-Kähler manifold if $d\omega=0$.
    \end{enumerate}
In such cases, we call $\omega$ a Kähler (respectively quasi-Kähler and almost-Kähler) metric on $(M,J)$.
\end{definition}

For Kähler manifolds, the almost-complex structure $J$ is compatible with the Levi-Civita connection $D$ (i.e., $DJ=0$), making the Levi-Civita connection suitable for studying the complex properties of the manifold. This does not happen for non-Kähler manifolds. Therefore, it is common to use different connections. \\
Gauduchon \cite{gauduchon1997Hermitian} defined a family of connections $\nabla^t$, $t \in \mathbb{R}$, for any almost-Hermitian manifold $(M,g,J)$, that are compatible with both the metric and the almost-complex structure. Here we will be interested in the Chern connection $\nabla^1$, which we will denote simply by $\nabla$. The Chern connection is given by 
   \begin{eqnarray*}\label{tconnection}
         g(\nabla_X Y , Z) &=& g(D_X Y ,Z)  +\frac{1}{2}(d^c\omega)^+ (X,JY,JZ) 
         - g(X,N(Y,Z)) \\ 
         &+& \frac{1}{2}(d^c \omega)^- (X,Y,Z),
   \end{eqnarray*}
where $d^c \omega (X,Y,Z) =- d\omega (JX,JY,JZ)$ and $d^c \omega^+$ means the $(1,2)+(2,1)$ part of $d^c \omega$ and $d^{c} \omega^-$ means the $(3,0)+(0,3)$ part of $d^c \omega$.\\

The next result \cite{gauduchon2010calabi}, gives an alternative formula for the Chern connection.

\begin{theorem}\label{ChernT}
    Let $(M,g,J)$ be an almost-Hermitian manifold. The Chern connection $\nabla$ satisfy
    \begin{equation}\label{ChernE}
\begin{aligned}
& g(\nabla_X Z, Y) = \frac{1}{2}X g(Z,Y) +\frac{1}{2} JX g(Z,JY)\\
&+\frac{1}{4}g( [X,Z] +[JX,JZ] + J[JX,Z] -J[X,JZ],Y)\\
&-\frac{1}{4} g([X,Y]+ [JX,JY] +J[JX,Y] - J[X,JY],Z).
\end{aligned}    
    \end{equation}
\end{theorem}

Now, we recall some of the different notions of curvature positivity. These notions are usually associated with Hermitian vector bundles but, here we will be interested only in the complexified tangent bundle, and we will allow our manifold to be almost-Hermitian.

\begin{definition}
    Let $(M,g,J)$ be an almost-Hermitian manifold. Denote by $R$ the curvature tensor of the Chern connection, we say that $(M,g,J)$ is
    \begin{itemize}
        \item [1)] Griffiths positive if, for any nonzero vectors $u,v \in T^{1,0}M$, $u=\sum u^i \frac{\partial}{\partial z_i}$ and $v=\sum v^i \frac{\partial}{\partial z_i}$, we have
        \begin{equation}
            R(u, \overline{u} , v , \overline{v}) = \sum_{ijkl}R_{i\overline{j}k\overline{l}}\ u^i \overline{u}^j v^k \overline{v}^l >0.
        \end{equation}

\item[2)] Nakano positive if, for any nonzero vector\\ $u =\sum_{ij} u^{ij} \frac{\partial}{\partial z_i} \otimes \frac{\partial}{\partial z_j} \in T^{1,0}M \otimes T^{1,0}M$, we have
\begin{equation}
    \sum_{ijkl} R_{i\overline{j}k\overline{l}} u^{ij} \overline{u}^{kl}>0. 
\end{equation}

\item[3)] Dual-Nakano positive if, for any nonzero vector\\ $u =\sum_{ij} u^{ij} \frac{\partial}{\partial z_i} \otimes \frac{\partial}{\partial z_j} \in T^{1,0}M \otimes T^{1,0}M$, we have
\begin{equation}
    \sum_{ijkl} R_{i\overline{j}k\overline{l}} u^{il} \overline{u}^{kj}>0. 
\end{equation}
    \end{itemize}
The notions of semi-positivity, negativity and semi-negativity follow analogously for all the concepts above.
\end{definition}
Also note that both Nakano and Dual-Nakano positivity (respectively semi-positivity, …) implies Griffiths positivity (respectively semi-positivity, …).\\

\section{Submersion metrics}

Let $M=G/K$ be a homogeneous space. A Riemannian metric $g$ on $M$ is called invariant when $G$ acts on it by isometries, that is, the map
\begin{align*}
    &\tau_a :G/K \rightarrow G/K\\
    & \tau_a (gK) =agK 
\end{align*}
satisfies $g(D\tau_a X ,D\tau_a Y)=g(X,Y)$, $\forall a \in G$ and $X,Y\in \mathfrak{X}(M)$.\\
If $M=G/K$ is reductive and
\begin{equation*}
    \mathfrak{g} = \mathfrak{k} \oplus \mathfrak{m}
\end{equation*}
is its reductive decomposition, then the set of invariant metrics, $\mathcal{M}^{inv} (M)$, is in bijection with the set of inner products on $\mathfrak{m}$ invariant by the isotropy representation $Ad^{G/K}$.\\

 \par The space $\mathcal{M}^{inv} (M)$ might be empty for some homogeneous spaces, but there is another way to define metrics on $M= G/K$. Given a right (or left) invariant metric $g$ on $G$ there is a unique metric $\pi_* g$ such that the projection $\pi : G \rightarrow G/K$ is a Riemannian submersion.\\

 Note that, in general, submersion metrics are not invariant, but it turns out that if the initial metric $g$ on $G$ is bi-invariant, then $\pi_* g$ is invariant.
 \begin{proposition}
     Let $g$ be a bi-invariant metric on $G$. Then the submersion metric $\pi_* g$ is an invariant metric for $M=G/K$.  
 \end{proposition}
\begin{proof}

Let $\mathfrak{g}$ and $\mathfrak{k}$ be the Lie algebras of $G$ and $K$, respectively. Since the metric $g$ is bi-invariant, we take $\mathfrak{m}=\mathfrak{k}^{\perp}$ and we have a reductive decomposition
\begin{equation*}
    \mathfrak{g} = \mathfrak{k}\oplus\mathfrak{m}.
\end{equation*}

The horizontal space of the projection $\pi : G \rightarrow G/K$ at a point $x \in G$ is
\begin{equation*}
    Hor_x = (Ver_x )^\perp = ((DL_x )_e \mathfrak{k})^\perp = (DL_x )_e \mathfrak{m}.
\end{equation*}
Given $X,Y \in T_{xK} M$, there are $v,w \in \mathfrak{m}$ such that
\begin{align*}
    &X =\frac{d}{dt}\mid_{t=0} x e^{tv}K & Y =\frac{d}{dt}\mid_{t=0} x e^{tw}K .
\end{align*}
Then, it is easy to see that both $\pi_* g (X,Y)$ and $\pi_* g (\tau_a X , \tau_a Y)$ are equal to $g(v,w)$.
\end{proof}

Finally, we have the following result due to \cite{ustinovskiy2018hermitian}.
\begin{theorem}\label{Submersions}
Let $M=G/K$ be a complex homogeneous manifold, $J$ an integrable complex structure on $M$ and $\pi_* g$ a submersion metric induced from a right invariant metric $g$ on $G$. Then $(M, \pi_* g ,J )$ is dual-Nakano semi-positive.
\end{theorem}

\section{Flag manifolds}

\begin{definition}
A (complex) flag manifold, is a homogeneous space $M=G/K$, where $G$ is a compact semisimple lie group and $K=C(T)$ is the centralizer of a torus $T$. If $T$ is a maximal torus (i.e., $C(T)=T$) we call $M= G/T$ a maximal flag manifold.
\end{definition}

 Let $G$ be a semisimple Lie group and $\mathfrak{g}$ its Lie algebra. We have the root space decomposition

\begin{equation}
    \mathfrak{g}^{\mathbb{C}} = \mathfrak{t}^{\mathbb{C}} \oplus \sum_{\alpha \in R} \mathfrak{g}_\alpha,
\end{equation}
  where $\mathfrak{t}$ is the lie algebra of the maximal torus and $R$ is the set of roots. Fixed a set of simple root $\Pi$ for $R$, one can construct a flag manifold in the following way:

  Given $\Pi_K \subset \Pi$ we define $R_K = Span[\Pi_K] \cap R$ and the Lie algebra
  \begin{equation}
      \mathfrak{k}^{\mathbb{C}} = \mathfrak{t}^{\mathbb{C}} \oplus \sum_{\alpha\in R_K} \mathfrak{g}_\alpha .
  \end{equation}

 Then, setting $R_M = R/R_K$ we define
\begin{equation}
    \mathfrak{m}^\mathbb{C} = \sum_{\alpha\in R_M} \mathfrak{g}_\alpha .
\end{equation}
 Combining this equations, we get
 \begin{equation}
     \mathfrak{g}^\mathbb{C} = \mathfrak{k}^{\mathbb{C}} \oplus \sum_{\alpha \in R_M} \mathfrak{g}_\alpha =    \mathfrak{k}^{\mathbb{C}} \oplus \mathfrak{m}^{\mathbb{C}}.
 \end{equation}

Now, we can use a compact real form to get a reductive decomposition
\begin{equation}
    \mathfrak{g} = \mathfrak{k} \oplus\mathfrak{m}.
\end{equation}
The Lie subgroup $K \subset G$ whose lie algebra is $\mathfrak{k}$ is the centralizer of a torus. Hence, we have defined a flag manifold $G/K$ from a choice of simple roots $\Pi_K \subset \Pi$, whose tangent space at the origin $o=eK$ can be identified with $\mathfrak{m}$. Also, we call $R_M$ the set of roots on the flag manifold.\\ 

\begin{remark}
    Note that a flag manifold is maximal if, and only if, $\Pi_K = \{\emptyset\}$.
\end{remark}

A flag manifold, defined in such a way, is usually represented by a painted Dynkin diagram, by painting the roots $\Pi_K$ of the Dynkin diagram of $G$.

\begin{example}\label{exemplo1}
Let $G=SU(4)$ be the Dynkin diagram of $SU(4)$ is
\begin{equation}
      \begin{dynkinDiagram}[scale=2, mark=*]A{3}
 \dynkinRootMark{o}1
 \dynkinRootMark{o}2
 \dynkinRootMark{o}3
 \end{dynkinDiagram}.
\end{equation}
   The roots of $SU(4)$ are $\{ \alpha_{ij} | i,j=1,2,3,4  \ , \ i\neq j\}$ where\\ $\alpha_{ij} ( Diag(a_1 , a_2 , a_3 , a_4))= a_i - a_j$. A set of positive roots is $\{\alpha_{i,j} | i <j\}$, and the associated set of simple roots is $\Pi =\{\alpha_{12} , \alpha_{23} , \alpha_{34}\}$.\\
   Any choice of $\Pi_k \subset \Pi$ determines a flag manifold, so there are $7$ choices for flag manifolds associated with $SU(4)$, although different choices might lead to the same flag manifold.\\
For example, the choice $\Pi_k = \{\alpha_{12},\alpha_{34}\}$, represented by the painted Dynkin diagram
\begin{equation}
      \begin{dynkinDiagram}[scale=2, mark=*]A{3}
 \dynkinRootMark{o}2
 \end{dynkinDiagram},
\end{equation}
is the Grassmannian $Gr_2 (\mathbb{C}^4) = SU(4)/S( U(2) \times U(2))$.\\
    Furthermore, the choice $\Pi_k = \{\alpha_{12}, \alpha_{23}\}$, represented by the painted Dynkin diagram
\begin{equation}
      \begin{dynkinDiagram}[scale=2, mark=*]A{3}
 \dynkinRootMark{o}3
 \end{dynkinDiagram},
\end{equation}
is the complex projective space $ \mathbb{C}P^{3} = SU(4)/S( U(3) \times U(1))$.
\end{example}

It turns out that this construction of flag manifolds through painted Dynkin diagrams exhausts all flag manifolds \cite{alekseevsky1997flag}.
\begin{theorem}
    There is a one-to-one correspondence between flag manifolds (up to isomorphisms of homogeneous spaces) and painted Dynkin diagrams (up to equivalence).
\end{theorem}

Next, we use some theory of complex lie algebras to create a suitable base to $\mathfrak{m}^\mathbb{C}$. Let $\mathfrak{g}^{\mathbb{C}}$ be a complex semisimple lie algebra, $B$ be the killing form, $\mathfrak{t}^{\mathbb{C}}$ the maximal torus and $R$ the set of roots. For every $\alpha \in R$ denote by $H_\alpha \in \mathfrak{t}^\mathbb{C}$ the unique vector that satisfies $B(H_\alpha , Y) = \alpha(Y)$, $\forall Y \in \mathfrak{t}^\mathbb{C}$.

\begin{proposition}\label{Liealgebras}
For every $\alpha \in R$, there is an element $X_\alpha \in \mathfrak{g}_\alpha$ such that
\begin{itemize}
    \item [1)] $X_\alpha$ spans $\mathfrak{g}_\alpha$,
    \item[2)] $B(X_\alpha , X_\beta)=0$ for $\beta\neq-\alpha$ and $B(X_{\alpha},X_{-\alpha})=1$, 
    \item[3)] $[X_\alpha , X_{-\alpha}] = H_\alpha$,
    \item[4)] $[X_\alpha , X_\beta] = m_{\alpha , \beta} X_{\alpha+\beta}$ for some real number $m_{\alpha , \beta}$ called the constant structures (if $\alpha+\beta \notin R$, then $m_{\alpha,\beta}=0$, so the equation makes sense).
\end{itemize}
And the constant structures $m_{\alpha , \beta}$ satisfy 
\begin{itemize}
    \item [a)] $   m_{\alpha,\beta} = - m_{\beta,\alpha} .$
    \item[b)] $ m_{\alpha,\beta}= - m_{-\alpha,-\beta}$
    \item[c)] $m_{\alpha , \beta} = m_{\beta , -\alpha-\beta} = m_{-\alpha-\beta,\alpha}$
\end{itemize}
\end{proposition}

Thus, if we have a flag manifold $M=G/K$ determined by $\Pi_K \subset \Pi$, then 
\begin{equation*}
    \mathfrak{m}^\mathbb{C} = \sum_{\alpha \in R_M} \mathbb{C}X_\alpha.
\end{equation*}

Now, we need to understand the isotropy representation of flag manifold
\begin{equation}
    Ad^{G/K} = Ad^K |_\mathfrak{m}: \mathfrak{m} \rightarrow \mathfrak{m}.
\end{equation}

This representation decomposes $\mathfrak{m}$ into irreducible invariant subspaces
\begin{equation}
    \mathfrak{m} = \mathfrak{m}_1 \oplus ... \oplus \mathfrak{m}_k.
\end{equation}

To understand the subspaces $\mathfrak{m}_i$, we need to look at the complexification of the isotropy representation $Ad^{G/K} : \mathfrak{m}^\mathbb{C} \rightarrow \mathfrak{m}^\mathbb{C}$.

\begin{proposition}
    Given a flag manifold $G/C(T)$ defined by $\Pi_K \subset \Pi$, define an equivalence relation on $R_M$ by setting $\alpha \sim \beta$ if $(\alpha-\beta) (T) \equiv 0$, or equivalently $\alpha- \beta \in \mathbb{Z} \Pi_K$.
    Then, there is a one-to-one correspondence between equivalence classes $[\alpha]$ on $R_M$ and irreducible invariant subspaces of $m^\mathbb{C}$ by the isotropy representation. The correspondence is given by
    \begin{equation}
        [\alpha] \mapsto \bigoplus_{\beta \sim \alpha} g_{\beta}.
    \end{equation}
And the irreducible invariant subspaces of $\mathfrak{m}$ are
\begin{equation}\label{A}
    \mathfrak{m}_i = \bigoplus_{\beta \sim \alpha}  \mathfrak{m}_\beta
\end{equation}
    where $\mathfrak{m}_\beta = \mathfrak{m} \cap (\mathbb{C} X_\beta \oplus \mathbb{C}X_{-\beta})$.
\end{proposition}

\begin{remark}
    We may write the isotropy representation of a flag manifold as
    \begin{equation}
        \mathfrak{m} = \Big( \mathfrak{m}_{\alpha_1} \oplus ... \oplus \mathfrak{m}_{\alpha_k}\Big) \oplus \Big(\mathfrak{m}_{\beta_1} \oplus ... \oplus \mathfrak{m}_{\beta_l}\Big) \oplus ... \oplus \Big(\mathfrak{m}_{\gamma_1} \oplus ... \oplus \mathfrak{m}_{\gamma_m}\Big)
    \end{equation}
    to indicate that each of the spaces\\ $ \Big( \mathfrak{m}_{\alpha_1} \oplus ... \oplus \mathfrak{m}_{\alpha_l}\Big) , \Big(\mathfrak{m}_{\beta_1} \oplus ... \oplus \mathfrak{m}_{\beta_l} \Big) ,..., \Big(\mathfrak{m}_{\gamma_1} \oplus ... \oplus \mathfrak{m}_{\gamma_m}\Big)$  is an irreducible invariant subspace.
\end{remark}

\subsection{Invariant metrics and almost-complex structures}

A $G$-invariant metric on a flag manifold $M=G/K$ is in one-to-one correspondence with $Ad^{G/K}$-invariant inner products on $\mathfrak{m}$. If
\begin{equation}
      \mathfrak{m} = \mathfrak{m}_1 \oplus ... \oplus \mathfrak{m}_k
\end{equation}
is the isotropy decomposition, then, any such inner product $\lambda$ is of the form
\begin{equation}
    \lambda = -\lambda_1 B|_{\mathfrak{m_1}} -... -\lambda_k B|_{\mathfrak{m_k}},
\end{equation}
where $B$ is the Cartan-Killing form of $G$. Therefore, any invariant metric on $G/K$ can be identified in such way with the set of ordered positive real numbers $(\lambda_1 ,..., \lambda_k )$.\\

\begin{remark}
Alternatively, using Equation \ref{A}, we can identify an invariant metric $\lambda$ with a set of positive real numbers $\{\lambda_\alpha\}_{\alpha \in R_M }$ provided $\lambda_{\alpha} = \lambda_{-\alpha}$ and $\lambda_\alpha = \lambda_ \beta$ if $\alpha \sim \beta$. We do so by seting
\begin{equation}
    \lambda_\alpha = \lambda(X_\alpha ,\overline{X_\alpha}).
\end{equation}
\end{remark}

Now, given an almost-complex structure $J$ on $M=G/K$,  we call it $G$-invariant if $J$ commutes with the action of $G$. Invariant almost-complex structures are in one-to-one correspondence with maps $J : \mathfrak{m} \rightarrow \mathfrak{m}$ such that\\ $J^{2}=-Id$ and $J$ commutes with the isotropy representation.\\

We can extend  $J$ to a map $J: \mathfrak{m}^\mathbb{C} \rightarrow \mathfrak{m}^\mathbb{C}$ and by the $Ad^{G/K}$-invariance we get that for every root $\alpha \in R_M$
\begin{equation}
    J X_\alpha = \pm \sqrt{-1}X_\alpha.
\end{equation}
So, we define $\epsilon_\alpha$ to be the number, $1$ or $-1$ such that
\begin{equation}
    J X_\alpha = \epsilon_\alpha\sqrt{-1} X_\alpha,
\end{equation}
since $J^2 = -Id$ we get that $\epsilon_{-\alpha} = - \epsilon_\alpha$.
Hence, we get the following
\begin{proposition}
    Let $M=G/K$ be a flag manifold. The set of invariant almost-complex structures is in one-to-one correspondence with a set of numbers $\{\epsilon_\alpha\}_{\alpha \in R_M} \subset \{-1, 1\}$ such that $\epsilon_{-\alpha} = -\epsilon_{\alpha}$ and $\epsilon_\alpha = \epsilon_\beta$ if $\alpha \sim \beta$.
\end{proposition}

\begin{remark}
    We also note that on a flag manifold $M=G/K$ endowed with an invariant almost-complex structure $J$, any invariant metric $\lambda$ is an almost-Hermitian metric on $(M,J)$.
\end{remark}

\begin{definition}
    Given a flag manifold $M=G/K$ endowed with an invariant almost-complex structure $J: \mathfrak{m} \rightarrow\mathfrak{m}$, we define the set of positive roots $R_M^+$ (related to $J$) as
    \begin{equation}
        R_M^+ = \{\alpha \in R_M \ | \ \epsilon_\alpha = 1\}. 
    \end{equation}
    Note that this might not define a set of positive roots in the Lie algebraic sense.\\
\end{definition}

\begin{remark}
    Using the above definition, given an almost-complex structure $J$ on a flag manifold $M=G/K$, we can identify the complexified tangent space with $T^{1,0} \mathfrak{m} = \bigoplus_{\alpha \in R_{M}^{+}}\mathbb{C}X_\alpha$ and the anti-complexified tangent space with $T^{0,1} \mathfrak{m} = \bigoplus_{\alpha \in -R_{M}^{+}}\mathbb{C}X_\alpha$.
\end{remark}

\begin{remark}
    In what follows an almost-complex structure $J$ and a metric $\lambda$, will always be understood to be invariant unless otherwise stated. 
\end{remark}

We can treat invariant metrics $\lambda$ and invariant almost-complex structures $J$ as set of numbers $\{\lambda_\alpha\}$, $\{\epsilon_\alpha\}$. And by using the basis $\{X_\alpha\}$ of $\mathfrak{m}^{\mathbb{C}}$ discussed in Proposition \ref{Liealgebras}, we can translate the geometrical properties discussed in section \ref{Almost-complex} to algebraic properties on these set of numbers. In fact, we have the following results (\cite{arvanitogeorgos2003introduction}, \cite{san2003invariant}).

\begin{proposition}\label{integrable}
    Given a flag manifold $M=G/K$ endowed with an almost-complex structure $J$ then, $J$ is integrable if, and only if, for every pair $\alpha,\beta \in R_M^+$, if $\alpha+\beta \in R_M$ we have $\alpha+\beta \in R_M^+$.
\end{proposition}

\begin{proposition}\label{kahler}
Given a flag manifold $M=G/K$ endowed with an almost-complex structure $J$ and a metric $\lambda$ then, $(M,\lambda,J)$ is a Kähler manifold, if and only if,
\begin{equation}
    \epsilon_\alpha \lambda_\alpha +\epsilon_\beta \lambda_\beta = \epsilon_{\alpha+\beta}\lambda_{\alpha+\beta},
\end{equation}
for every pair of roots $\alpha, \beta \in R_M$, whenever $\alpha+\beta \in R_M$.
\end{proposition}

\begin{proposition}\label{quasi--kahler}
    Given a flag manifold $M=G/K$ endowed with an almost-complex structure $J$ and a metric $\lambda$ then, $(M,\lambda,J)$ is a quasi-Kähler manifold, if and only if, for every pair of roots $\alpha, \beta \in R_M^+$ such that $\alpha + \beta \in R_M^+$ we have
    \begin{equation}
        \lambda_\alpha +\lambda_\beta = \lambda_{\alpha+\beta}.
    \end{equation}
\end{proposition}

The last result is about almost-Kähler flag manifolds.

\begin{proposition}
    Given a flag manifold $M=G/K$ endowed with an almost-complex structure $J$ and a metric $\lambda$. If $(M,\lambda,J)$ is almost-Kähler then, it is Kähler.
\end{proposition}

\subsection{Examples}

Finally, we consider two examples of flag manifolds, each a realization of complex projective space as a homogeneous space.  Moreover, by the classification in \cite{ziller1982homogeneous}, these are the only two realizations of \(\mathbb{C}P^{n-1}\) as a homogeneous space.

\begin{enumerate}
    \item [1)] $\mathbb{C}P^{n-1} = SU(n)/S( U(n-1) \times U(1))$\\

Similar to what we have seen in Example \ref{exemplo1}, the roots of the semisimple group $SU(n)$ are $R = \{\alpha_{ij} \ | \ i \neq j \} $ where
\begin{equation}
    \alpha_{ij} (\rm{Diag}(a_1 ,..., a_n ) = a_i -a_j .
\end{equation}
A set of simple roots is $\Pi = \{\alpha_{12} ,..., \alpha_{n-1 , n}\}$ and the associated Dynkin diagram $A_n$ is
\begin{equation}
      \begin{dynkinDiagram}[scale=2, mark=*]A{}
 \dynkinRootMark{o}1
  \dynkinRootMark{o}2
   \dynkinRootMark{o}3
    \dynkinRootMark{o}4
 \end{dynkinDiagram}.
\end{equation}

By choosing $\Pi_K = \{\alpha_{12} ,..., \alpha_{n-2,n-1}\}$ (or equivalently $\Pi_K = \{\alpha_{23} ,..., \alpha_{n-1,n}\}$) we get $\mathbb{C}P^{n-1}=SU(n)/S( U(n-1) \times U(1))$ with its associated painted Dynkin diagram 
\begin{equation}
      \begin{dynkinDiagram}[scale=2, mark=*]A{}
 \dynkinRootMark{o}4
 \end{dynkinDiagram}.
\end{equation}

The set of roots on the flag is $R_M = \{\alpha_{ij} \ | \ i=1\  \text{or} \ j=1\}$. Therefore, the isotropy decomposition has only one summand and hence, admits a unique integrable complex structure $J$, by setting $\epsilon_{\alpha_{1j}} = 1$ and $\epsilon_{\alpha_{j1}} = -1$. It also admits a unique (up to parametrization) metric $\lambda = -B$. This metric is the Fubini study metric and $(\mathbb{C}P^{n-1} , \lambda , J)$ is a Kähler-Einstein manifold.

\item[2)] $\mathbb{C}P^{2n-1} = Sp(n)/Sp(n-1) \times U(1)$.\\

First, we will analyze the Lie algebra of $Sp(n)$ which is the $C_n$ Lie algebra, for more details see \cite{san1999algebras}.

The lie algebra $\mathfrak{sp}(n)$ is composed by the $2n\times2n$ matrices of the form
\begin{equation}
    \left(\begin{array}{cc}
        \alpha & \beta \\
        \gamma & -\alpha^t
    \end{array}\right),
\end{equation}
where $\beta $ and $\gamma$ are symmetric matrices. The Cartan subalgebra is\\
$\mathfrak{h}=\left(\begin{array}{cc}
\Lambda&\\
&-\Lambda
\end{array}\right)$, where $\Lambda=\rm{Diag} (x_1 ,... x_n )$.

Consider the linear functionals given by
\begin{equation}
\lambda_{i}:\Lambda=\rm{Diag}\{x_1 ,..., x_n\}\longrightarrow x_i   .  
\end{equation}
The space of roots $R$ is composed by the following roots
\begin{enumerate}
    \item $\pm(\lambda_i - \lambda_j )$ for $i \neq j$,
    \item $\pm(\lambda_i + \lambda_j )$ for $i \neq j$,
    \item $\pm2\lambda_i$, 
\end{enumerate}
where $i$ and $j$ varies through $1,...,n$.\\
A set of simple roots is
\begin{equation}
    \Pi = \{ \lambda_1 - \lambda_2 , ..., \lambda_{n-1} - \lambda_n , 2\lambda_n \},
\end{equation}
the corresponding Dynkin diagram is
\begin{equation}
     \begin{dynkinDiagram}[scale=2, mark=*]C{}
 \dynkinRootMark{o}1
  \dynkinRootMark{o}2
   \dynkinRootMark{o}3
    \dynkinRootMark{o}4
      \dynkinRootMark{o}5
 \end{dynkinDiagram}.
\end{equation}
  and the roots vectors, described in Proposition \ref{Liealgebras} are

 \begin{equation}
     \begin{aligned}
         & X_{\lambda_i - \lambda_j } = \frac{1}{2\sqrt{n+1}}(E_{ij} - E_{n+j,n+i}),\\
         &X_{\lambda_i + \lambda_j } = \frac{1}{2\sqrt{n+1}}(E_{i,n+j} + E_{j,n+i}),\\
         &X_{-\lambda_i - \lambda_j } = \frac{1}{2\sqrt{n+1}}(E_{n+i,j} + E_{n+j,i}),\\
         &X_{2\lambda_i   } = \frac{1}{\sqrt{2(n+1)}}E_{i,n+i}, \\   &X_{-2\lambda_i   } = \frac{1}{\sqrt{2(n+1)}}E_{n+i,i}, \\
     \end{aligned}
 \end{equation}
where $E_{i,j}$ is the $2n\times2n$ matrix with zero in all coordinates except the $i,j$-th coordinate, which equals one.

Now, the flag manifold determined by the choice\\ $\Pi_K = \{ \lambda_2 - \lambda_3 , ..., \lambda_{n-1} - \lambda_n , 2\lambda_n \}$ is $\mathbb{C}P^{2n-1} = Sp(n)/Sp(n-1) \times U(1)$ and the corresponding painted Dynkin diagram is
\begin{equation}
      \begin{dynkinDiagram}[scale=2, mark=*]C{}
  \dynkinRootMark{o}1
      \end{dynkinDiagram}.
\end{equation}

The set of roots $R_M$ is composed by
\begin{equation}
    R_M = \pm\{(\lambda_1 - \lambda_j ) , (\lambda_1 + \lambda_j ), 2\lambda_1   \ | \ j =2,...,n \}.
\end{equation}

The isotropy representation is given by
\begin{equation}
    \mathfrak{m} = \Big( \bigoplus_j \mathfrak{m}_{\lambda_1 - \lambda_j} \bigoplus_j \mathfrak{m}_{\lambda_1 +\lambda_j} \Big) \oplus \mathfrak{m}_{2\lambda_1},
\end{equation}
which has $2$ isotropic summands. Therefore, any invariant metric on it can be expressed as a pair $(a,b)$ for $a,b >0$, or to simplify we will rescalonate the metric to be of the form $(1,t)$.\\

There are also two non-equivalent almost-complex structures. We denote by $J_1 =(+,+)$ the one that sets $\epsilon_{\lambda_1 - \lambda_j} = \epsilon_{\lambda_1 +\lambda_j }= \epsilon_{2\lambda_1} =1$ and we denote by $J_2 =(+,-)$ the almost-complex structure given by $\epsilon_{\lambda_1 - \lambda_j} = \epsilon_{\alpha_1 +\lambda_j }=1$ and $\epsilon_{2\lambda_1} =-1$.\\

Since $(\lambda_1 - \lambda_j )  +(\lambda_1 +\lambda_j ) = 2\lambda_1$ from Propositions \ref{integrable} and \ref{kahler} we get that the almost-complex structure $J_1 = (+,+)$ is integrable and $J_2 = (+,-)$ is non-integrable. The metric $(1,2)$ is the unique Kähler metric for $J_1 =(+,+)$. Furthermore, for $J_2=(+,-)$ we note that there are no $\alpha, \beta \in R_M^{+}$ such that $\alpha+\beta \in R_M^{+}$ so, from  Proposition \ref{quasi--kahler} we get that any metric $(1,t)$ is quasi-Kähler.

Lastly, we note that the unique Kähler metric $\lambda = (1,2)$ on\\ $(\frac{Sp(n)}{Sp(n-1) \times U(1)} , J_1 )$ is the same as the Fubini-study metric. This happens because, as we have seen, $Sp(n)/Sp(n-1) \times U(1)$ has a unique $Sp(n)$-invariant integrable complex structure and has a unique $Sp(n)$-invariant Kähler metric associated with it. Since $Sp(n) \subset SU(2n)$, the Fubini-Study metric and its associated complex structure are also $Sp(n)$-invariant and, therefore, equal to $J_1 = (+,+)$ and $\lambda= (1,2)$.

\end{enumerate}

\section{Curvature positivity on Kähler flag manifolds}

In this section, we compile a series of results to classify all invariant Kähler metrics on flag manifolds that have positive curvature.\\

Let $M=G/K$ be a flag manifold, endowed with an integrable complex structure $J$. Theorem \ref{Submersions} gives us a way to construct a metric on $M$, that is dual-Nakano semi-positive, from bi-invariant metrics on $G$. Since $G$ is semisimple, the only bi-invariant metric (up to parametrization) is the negative of the killing form $B$. Thus, we have the following:

\begin{proposition}\label{lambda=1}
Let $M=G/K$ be a flag manifold, endowed with an integrable complex structure $J$. Let
\begin{equation*}
    \mathfrak{g} = \mathfrak{k} \oplus \mathfrak{m}
\end{equation*}
be the reductive decomposition, then the metric $-B|_{\mathfrak{m}}$ on $(M,J)$ is dual-Nakano semi-positive.\\
In other words, on every flag manifold endowed with a complex structure, the metric $\lambda$ given by $\lambda_\alpha = 1$, for all $\alpha \in R_M$ is a dual-Nakano semi-positive metric.
    
\end{proposition}
Next, we want to classify in which flag manifolds the metric given by $\lambda_\alpha = 1$, for all $\alpha \in R_M$, is Kähler.\\
Due to Proposition \ref{kahler}, if there are roots $\alpha, \beta$ and $ \alpha+\beta \in R_M$ such that $\lambda_\alpha = \lambda_\beta =1$, then setting $\lambda_{\alpha+\beta}=1$ does not satisfy the Kähler condition. Therefore, this metric is Kähler if, and only if, there are no roots on $R_M$ that can be summed up to a root on $R_M$, which happens if, and only if, $\Pi_M$ consist of only one root of mark one. A list of all such spaces is given by the table below.

\begin{table}[htbp]
  \centering
  \caption{Flag manifolds with a single isotropy summand}
  \label{tab:flag-manifolds}
  
  \begin{tabular}{|
      >{\centering\arraybackslash}m{0.45\linewidth}
      | >{\centering\arraybackslash}m{0.45\linewidth}
    |}
    \hline
    \textbf{$G/K$}
      & \textbf{Painted Dynkin diagram} \\
    \hline
\hline
    $\dfrac{SU(n)}{S\bigl(U(n-k)\times U(k)\bigr)}$
      &
      \adjustbox{valign=c}{%
        \begin{dynkinDiagram}[scale=1.5, mark=*]A{}
          \dynkinRootMark{o}3
        \end{dynkinDiagram}%
      } \\
    \hline

    $\dfrac{SO(2n+1)}{SO(2n-1)\times U(1)}$
      &
      \adjustbox{valign=c}{%
        \begin{dynkinDiagram}[scale=1.5, mark=*]B{}
          \dynkinRootMark{o}1
        \end{dynkinDiagram}%
      } \\
    \hline

    $\dfrac{Sp(n)}{U(n)}$
      &
      \adjustbox{valign=c}{%
        \begin{dynkinDiagram}[scale=1.5, mark=*]C{}
          \dynkinRootMark{o}5
        \end{dynkinDiagram}%
      } \\
    \hline

    $\dfrac{SO(2n)}{U(n)}$
      &
      \adjustbox{valign=c}{%
        \begin{dynkinDiagram}[scale=1.5, mark=*]D{}
          \dynkinRootMark{o}5
        \end{dynkinDiagram}%
      } \\
    \hline

    $\dfrac{SO(2n)}{SO(2n-2)\times U(1)}$
      &
      \adjustbox{valign=c}{%
        \begin{dynkinDiagram}[scale=1.5, mark=*]D{}
          \dynkinRootMark{o}1
        \end{dynkinDiagram}%
      } \\
    \hline

    $\dfrac{E_6}{SO(10)\times U(1)}$
      &
      \adjustbox{valign=c}{%
        \begin{dynkinDiagram}[scale=1.5, mark=*]E{6}
          \dynkinRootMark{o}6
        \end{dynkinDiagram}%
      } \\
    \hline

    $\dfrac{E_7}{E_6 \times U(1)}$
      &
      \adjustbox{valign=c}{%
        \begin{dynkinDiagram}[scale=1.5, mark=*]E{7}
          \dynkinRootMark{o}7
        \end{dynkinDiagram}%
      } \\
    \hline
  \end{tabular}
\end{table}

\begin{remark}
    Let $\gamma$ be a maximal root of a semisimple Lie algebra  $\mathfrak{g}$, given a system of simple roots $\{E_i\}$, one can write
    \begin{equation}
        \gamma=a_1 E_1 +...+a_n E_n
    \end{equation}
    for some $a_i >0$ and we call the number $a_i$ the mark of $E_i$.
\end{remark}
Turns out that these are precisely the compact Hermitian symmetric spaces, and more surprisingly they cover all Griffiths semi-positive Kähler metrics on flag manifolds. We have the generalized Frankel's conjecture proved by \cite{mok1988uniformization} that states the following:

\begin{theorem}\label{Mok}
    Let $(X,g)$ be a compact Kähler manifold with Griffiths semi-positive curvature, let $(\widetilde{X},\widetilde{g})$ be its universal covering space, then $(\widetilde{X},\widetilde{g})$ is isometrically biholomorphic to
    \begin{equation}
        (\mathbb{C}^k , g_0 ) \times (\mathbb{C}P^{N_1},\theta_1 ) \times ...\times (\mathbb{C}P^{N_l},\theta_l ) \times (M_1 , g_1 ) \times ... \times (M_p , g_p )
    \end{equation}
for some $p,k, N_1 ,..., N_l \in \mathbb{N}$, where $g_0$ is the Euclidean metric on $\mathbb{C}^k$, $\theta_i$ are Kähler metrics on $\mathbb{C}P^{N_i}$ with Griffiths semi-positive curvature, $M_i$ are compact Hermitian symmetric spaces and $g_i$ are its canonical metrics.
\end{theorem}

Applying this theorem for compact and simply connected spaces, we conclude:
\begin{corollary}\label{FlagsMok}
    Let $(X,g)$ a compact, simply connected Kähler manifold with Griffiths semi-positive curvature then it is isometrically biholomorphic to
\begin{equation*}
    (\mathbb{C}P^{N_1},\theta_1 ) \times ... \times (\mathbb{C}P^{N_l},\theta_l ) \times (M_1 , g_1 ) \times ... \times (M_p , g_p ).
\end{equation*}
    In particular $X$ is a flag manifold.
\end{corollary}
\begin{remark}
Note that on the above corollary, although $M$ is a flag manifold, the metric $g$ might not be invariant.
\end{remark}

The natural follow up is to look for Griffiths semi-positive metrics on \(\mathbb{C}P^n\). The unique \(SU(n)\)-invariant complex structure on
\[
  \mathbb{C}P^n \;=\; \dfrac{SU(n)}{S( U(n-1) \times U(1))}
\]
admits the Fubini–Study metric, which is precisely the unique \(SU(n)\)-invariant Kähler metric on \(\mathbb{C}P^n\) appearing in Table \ref{tab:flag-manifolds}.

One could look for metrics on $\mathbb{C}P^{n}$ invariant by the action of another group. The only other realization of $\mathbb{C}P^{n}$ as a homogeneous space is by the action of $Sp(n)$, see \cite{ziller1982homogeneous},
\begin{equation}
    \mathbb{C}P^{2n-1} = \frac{Sp(n)}{Sp(n-1) \times U(1)}.
\end{equation}
This flag manifold has a unique $Sp(n)$-invariant complex structure $(+,+)$ and a unique $Sp(n)$-invariant Kähler metric $(1,2)$. But, this metric is the same as the Fubini-Study metric because $Sp(n+1) \subset SU(2n+2)$.

This leads us to the full classification of invariant Kähler metrics with Griffiths semi-positive curvature on flag manifolds.

\begin{corollary}
    Let $(X,g)$ a compact, simply connected Kähler manifold with Griffiths semi-positive curvature. If $g$ is invariant by the action of a transitive group then, it is isometrically biholomorphic to
\begin{equation*}
    (M_1 , g_1 ) \times ... \times (M_p , g_p ).
\end{equation*}
\end{corollary}

\section{Curvature positivity on quasi-Kähler metrics}

 The curvature positivity condition on Kähler flag manifolds was extremely restrictive. Only appearing on flag manifolds with one isotropic summand, where the underlying algebraic structure is particularly simple. If we require non-integrability of the almost-complex structures, we exclude all the simpler flag manifolds, making the situation seems even more restrictive.
 In this section, we correlate the algebraic structure of the root system of a flag manifold with the curvature positivity, and we use this to prove neither maximal flag manifolds nor flag manifolds associated with the classical Simple Lie groups admit quasi-Kähler metrics of semi-positive curvature on a non-integrable almost-complex structure. We also show that no flag manifold associated with $G_2$ admits such structures, and we conjecture that in fact no quasi-Kähler flag manifolds admits such a metric.\\

We will use Theorem \ref{ChernT} to compute an expression for the Chern connection and the curvature tensor for flag manifolds. Let $G/K$ be a flag manifold endowed with an almost-complex structure $J=\{\epsilon_\alpha \}$ and an invariant almost-Hermitian metric $\lambda= \{\lambda_\alpha \}$. Let $R_M$ be the set of roots on $G/K$. Given $\alpha, \beta \in R_M$, we want to compute the Chern connection $\nabla_{X_\alpha} X_\beta$. Remember that when we write $ \nabla_{X_\alpha}X_{\beta}$, we mean $(\nabla_{X^*_\alpha}X^*_\beta)_{eK}$, where $e$ is the identity element in $G$ and $X_{\alpha}^*$ is the Killing vector field generated by $X_\alpha$. From equation \ref{ChernE} we get

    \begin{equation}
\begin{aligned}
& \lambda(\nabla_{X_{\alpha}} {X_{\beta}}, {X_{\gamma}}) = \frac{1}{2}{X_{\alpha}} \lambda({X_{\gamma}},{X_{\beta}}) +\frac{1}{2} J{X_{\alpha}} \lambda({X_{\gamma}},J{X_{\beta}})\\
&-\frac{1}{4}\lambda( [{X_{\alpha}},{X_{\gamma}}] +[J{X_{\alpha}},J{X_{\gamma}}] + J[J{X_{\alpha}},{X_{\gamma}}] -J[{X_{\alpha}},J{X_{\gamma}}],{X_{\beta}})\\
&+\frac{1}{4} \lambda([{X_{\alpha}},{X_{\beta}}]+ [J{X_{\alpha}},J{X_{\beta}}] +J[J{X_{\alpha}},{X_{\beta}}] - J[{X_{\alpha}},J{X_{\beta}}],{X_{\gamma}}).
\end{aligned}    
    \end{equation}
The sign change in the last two terms happens because $[X_{\alpha}^*,X_{\beta}^*] = -[X_{\alpha},X_{\beta}]^*$. And, since we are dealing with Killing vector fields, the first two terms are null. Hence,
    \begin{equation}
\begin{aligned}
& \lambda(\nabla_{X_{\alpha}} {X_{\beta}}, {X_{\gamma}}) =\\
&-\frac{1}{4}m_{\alpha,\gamma}(1-\epsilon_\alpha \epsilon_\gamma -\epsilon_{\alpha+\gamma}\epsilon_\alpha + \epsilon_{\alpha+\gamma}\epsilon_\gamma )\lambda(X_{\alpha+\gamma},X_\beta )\\
&+\frac{1}{4}m_{\alpha,\beta} (1  - \epsilon_\alpha \epsilon_\beta -\epsilon_{\alpha+\beta} \epsilon_\alpha +\epsilon_{\alpha+\beta} \epsilon_\beta )\lambda(X_{\alpha+\beta}, X_\gamma).
\end{aligned}    
    \end{equation}
By using that $\lambda(X_{\alpha}, X_\gamma) = -\lambda_\gamma$  if $\gamma=-\alpha$ and $\lambda(X_{\alpha}, X_\gamma)=0$ otherwise, we conclude that, if $\gamma= -\alpha - \beta$  then,
   \begin{equation}
\begin{aligned}
& \lambda(\nabla_{X_{\alpha}} {X_{\beta}}, {X_{\gamma}}) =\\
&+\frac{1}{4}m_{\alpha,\gamma}(1-\epsilon_\alpha \epsilon_\gamma -\epsilon_{\alpha+\gamma}\epsilon_\alpha + \epsilon_{\alpha+\gamma}\epsilon_\gamma )\lambda_{\alpha+\gamma}\\
&-\frac{1}{4}m_{\alpha,\beta} (1  - \epsilon_\alpha \epsilon_\beta -\epsilon_{\alpha+\beta} \epsilon_\alpha +\epsilon_{\alpha+\beta} \epsilon_\beta )\lambda_{\alpha+\beta},
\end{aligned}    
    \end{equation}
and $\lambda(\nabla_{X_{\alpha}} {X_{\beta}}, {X_{\gamma}}) =0$ otherwise. Finally, by using that $m_{\alpha,-\alpha-\beta}= -m_{\alpha,\beta}$ we get
   \begin{eqnarray*}
 \nabla_{X_{\alpha}} {X_{\beta}} &=&
\frac{1}{4\lambda_{\alpha+\beta}}m_{\alpha,\beta} \Big(\lambda_{\beta}(1 +\epsilon_\alpha \epsilon_{\alpha+\beta} + \epsilon_\alpha \epsilon_\beta +\epsilon_\beta \epsilon_{\alpha+\beta})\\ &+&\lambda_{\alpha+\beta} (1 -\epsilon_\alpha \epsilon_\beta - \epsilon_\alpha \epsilon_{\alpha+\beta} +\epsilon_\beta \epsilon_{\alpha+\beta})\Big)X_{\alpha+\beta}.
    \end{eqnarray*}

Now we can compute the curvature tensor $R(X_\alpha , X_\beta, X_\gamma , X_\delta )$. Note that if $\delta \neq -\alpha - \beta-\gamma$, then $R(X_\alpha , X_\beta, X_\gamma , X_\delta )=0$. We will not need the full expression of the curvature on flag manifolds, but we will compute two particular cases.

Remember that for reductive homogeneous spaces, the curvature formula satisfies the following (\cite{kobayashi1996foundations}, Chapter X Proposition 2.3), for every\\ $X,Y,Z \in \mathfrak{m}$,
\begin{equation}
         R(X,Y,Z) = \nabla_X \nabla_Y Z - \nabla_Y \nabla_X Z - \nabla_{[X,Y]_\mathfrak{m}}Z - [[X,Y]_\mathfrak{k},Z].
\end{equation}

\begin{proposition}\label{Curvatura}
    Let $G/K$ be a flag manifold, $J$ an almost-complex structure and $\lambda$ an almost-Hermitian metric. Given $\alpha,\gamma \in R_{M}^{+}$, the curvature tensor $R$ of the Chern connection satisfies the following two identities.
\begin{equation}\label{Curvatura1}
    \begin{aligned}
       & R(X_\alpha , X_{-\alpha}, X_\gamma , X_{-\gamma}) =  \Big(  -\frac{\lambda_{\gamma-\alpha}}{ \lambda_{\gamma}} \tilde{m}_{\alpha , -\gamma}^2  \delta_{\epsilon_{ \gamma-\alpha}}^{1} +\\ &\frac{\lambda_\gamma }{ \lambda_{\alpha +\gamma}} \tilde{m}_{\alpha , \gamma}^2\delta_{\epsilon_{\alpha+\gamma}}^{1}
    -m_{\alpha,\gamma}^2 + m_{\alpha,-\gamma}^2 \Big)\lambda_\gamma 
    \end{aligned}
\end{equation}
and
    \begin{equation}
        \begin{aligned}
            R(X_\gamma , X_{-\alpha}, X_\alpha , X_{-\gamma} ) = -\frac{\lambda_\alpha \lambda_\gamma}{\lambda_{\alpha+\gamma}} m_{\alpha,\gamma}^2 \delta_{\epsilon_{\alpha+\gamma}}^{1} + m_{\alpha , -\gamma}^2 \xi_{\gamma,-\alpha},
        \end{aligned}
    \end{equation}
where
\begin{equation}
    \tilde{m}_{\alpha,\gamma} = 
    \begin{cases}
        &m_{\alpha,\gamma}, \ \ \text{if} \ \ \alpha+\gamma \in R_M\\
        &0, \ \ \text{otherwise}
    \end{cases}
\end{equation}
and
\begin{equation}
    \xi_{\gamma,-\alpha}=
      \begin{cases}
        &\lambda_\gamma \delta_{\epsilon_{\gamma-\alpha}}^{-1} + \lambda_\alpha \delta_{\epsilon_{\gamma-\alpha}}^1, \ \ \text{if} \ \ \gamma-\alpha \in R_M\\
        &\lambda_\gamma, \ \ \text{otherwise}.
    \end{cases}
\end{equation}
    
\end{proposition}

Note, on the above proposition, that the term $-m_{\alpha,\gamma}^2$, in the first Equation, imposes a strong restriction on the positivity of the curvature. In fact, we have the following.

\begin{lemma}\label{lemma}
    Let $(M,J,\omega)$ be a quasi-Kähler flag manifold. If there exists $\alpha,\gamma\in R_{M}^{+}$ such that $\alpha+\gamma \in R_M$ and $\alpha-\gamma \not\in R$, then $(M,J,\omega)$ is not Griffiths semi-positive.
\end{lemma}
\begin{proof}
It follows immediately from Equation \ref{Curvatura1} that
\begin{equation*}
    R(X_\alpha , X_{-\alpha}, X_\gamma, X_{-\gamma})=\Big( \frac{\lambda_\gamma }{ \lambda_{\alpha +\gamma}} m_{\alpha , \gamma}^2\delta_{\epsilon_{\alpha+\gamma}}^{1}
    -m_{\alpha,\gamma}^2 \Big)\lambda_\gamma.
\end{equation*}

If $\epsilon_{\alpha+\gamma}=-1$ then $\delta_{\epsilon_{\alpha+\gamma}}^{1}=0$ and $ R(X_\alpha , X_{-\alpha}, X_\gamma, X_{-\gamma})=-m_{\alpha,\gamma}^2 \lambda_\gamma  <0$.\\

Now, if $\epsilon_{\alpha+\gamma}=1$, due to the metric is quasi-Kähler, $\lambda_{\alpha+\gamma}=\lambda_\alpha +\lambda_\gamma$, and therefore, $\frac{\lambda_\gamma }{ \lambda_{\alpha +\gamma}}  <1$, which implies
\begin{equation}
      R(X_\alpha , X_{-\alpha}, X_\gamma, X_{-\gamma})=\Big( \frac{\lambda_\gamma }{ \lambda_{\alpha +\gamma}} m_{\alpha , \gamma}^2
    -m_{\alpha,\gamma}^2 \Big)\lambda_\gamma <0 
\end{equation}
and concludes the proof.
\end{proof}

Using this lemma we prove the next results about Griffiths semi-positive curvature.
\begin{theorem}\label{Height3}
    Let $M=G/K$ be a flag manifold, $J$ a non-integrable almost-complex structure and $\omega$ a quasi-Kähler metric. Suppose $(M,J,\omega)$ is Griffiths semi-positive, then $G$ has a simple root of mark at least $3$.
\end{theorem}
\begin{proof}
    Since $J$ is non-integrable, there are positive roots $\alpha,\gamma \in R_{M}^{+}$ such that $\alpha+\gamma \in R^{-}_M$. From Lemma
     \ref{lemma} we get that $\alpha-\gamma \in R$. \\
     Again using Lemma \ref{lemma} on the pair $-\alpha-\gamma$ and $\alpha$, which are both positive roots on $R_M$, whose sum lies on $R_M$, we get that $2\alpha+\gamma\in R$.\\
     To conclude the proof, choose a system of simple roots $\{
     E_i\}$ such that $\alpha$ and $\gamma$ are both positive roots in the Lie algebraic sense. Write $\alpha$ and $\gamma$ as a sum of these simple roots,
     
     \begin{equation}
     \begin{aligned}
          &\alpha= a_1  E_1 +...+a_nE_n , &\gamma=b_1 E_1 +...+b_n E_n
     \end{aligned}
     \end{equation}
      for some $a_i , b_i \geq 0$. Since $\alpha-\gamma\in R$ we have that, there is at least one simple root in common for $\alpha$ and $\gamma$ (i.e., there is a $i=1,...,n$ such that $a_i , b_i \neq 0$). As this is a common root on $\alpha$ and $\gamma$, then it appears at least three times on the root $2\alpha+\gamma$. 
\end{proof}
The Theorem above is very restrictive, i.e. none of the classical simple lie algebras $A_l , B_l , C_l$ and $D_l$ have a root of mark $3$ and therefore we have the following result.

\begin{corollary}
None flag manifold associated with $A_l , B_l , C_l$ and $D_l$ accepts a quasi-Kähler metric on a non-integrable almost-complex structure that is Griffiths semi-positive.    
\end{corollary}

\begin{theorem}\label{Maximal}
    Let $M$ be a maximal flag manifold different from $SU(2)/U(1)$. There are no almost-complex structure $J$ and quasi-Kähler metric $\omega$ such that $(M,J,\omega)$ has Griffiths semi-positive curvature.
\end{theorem}
\begin{proof}

If $J$ is integrable, then $R_M^+$ is a set of positive roots in the Lie algebraic sense. Therefore, there are simple roots $\alpha, \beta$ such that $\alpha+\beta \in R_M^+$. And since they are simple roots, $\alpha-\beta \not\in R_M$. By Lemma \ref{lemma} it cannot have a Griffiths semi-positive metric.

Now assume $J$ is non-integrable and $\omega$ is Griffiths semi-positive. Similar to the proof of Theorem \ref{Height3}, as the almost-complex structure is non-integrable, we can take $\alpha,\gamma \in R_{M}^{+}$ such that $\alpha+\gamma \in R_{M}^{-}$. Again, we use Lemma \ref{lemma} to guarantee the existence of the roots $\alpha-\gamma$, $\alpha+2\gamma$ and $2\alpha+\gamma$. Since the flag manifold is maximal, all these roots belong to $R_M$.\\
Exchanging $\alpha$ and $\gamma$ if necessary, we may assume that $\alpha-\gamma \in R_{M}^{+}$. We will divide the proof into three cases
\begin{enumerate}
    \item [1)]($2\alpha+\gamma \in R_{M}^{-}$)\\
In this case, applying the Lemma \ref{lemma} to the pair $\beta=\alpha-\gamma$ and $\delta=-2\alpha - \gamma$, we have that $\beta+\delta=-\alpha-2\gamma\in R$ and $\beta-\delta=3\alpha\notin R$.


\item[2)] ($\alpha+2\gamma \in R_{M}^{+}$ )\\
Here, we apply Lemma \ref{lemma} to the pair $\beta=\alpha+2\gamma$ and $\delta=\alpha-\gamma$. We note that $\beta+\delta=2\alpha+\gamma\in R$ and $\beta-\delta=3\gamma\notin R$.

\item[3)] ($2\alpha+\gamma \in R_{M}^{+}$ and $\alpha+2\gamma \in R_{M}^{-}$ )\\ 
In this case, we use Lemma \ref{lemma} on the pair $\beta=2\alpha+\gamma$ and $\delta=-\alpha-2\gamma$. We observe that $\beta+\delta=\alpha-\gamma\in R$ and $\beta-\delta=3\alpha+3\gamma\notin R$. 
\end{enumerate}
This contradicts the assumption of $(M,J,\omega)$ having Griffiths semi-positive curvature.
\end{proof}

The exceptional Lie algebras $G_2 , F_4 , E_6 , E_7$ and $E_8$ do have a root of mark $3$, but as we will see from the $G_2$ example, this condition does not guarantee the existence of such a metric.\\

\begin{proposition}
Let $M$ be a flag manifold associated with $G_2$ and $J$ a non-integrable almost-complex structure. There is no quasi-Kähler metric $\lambda$, such that $(M , J ,\lambda)$ is Griffiths semi-positive.
\end{proposition}
\begin{proof}
$G_2$ has two simple roots, which we denote by $\alpha$ and $\beta$. The set of roots, in terms of the simple roots, is composed by
\begin{equation*}
    R= \pm\{\alpha,\beta , \alpha+\beta , 2\alpha+\beta , 3\alpha+\beta , 3\alpha+2\beta\}.
\end{equation*}

There are $3$ flag manifolds associated with $G_2$:

i) by taking $\Pi_M = \{\emptyset\}$ we get the maximal flag $G_2/T^2$. 

ii) by taking $\Pi_M=\{\alpha\}$ we get $G_2/SU(2) \times U(1)$ where $SU(2) \subset G_2$ represents the long root. 

iii) by taking $\Pi_M=\{\beta\}$ we get the flag manifold $G_2/SU(2) \times U(1)$ where $SU(2) \subset G_2$ represents the short root.\\

Theorem \ref{Maximal} guarantees that are no quasi-Kähler metrics on the maximal flag $G_2/T^2$.\\

For the flag $G_2/SU(2) \times U(1)$ where $SU(2) \subset G_2$ represents the long root, the isotropy representation is
\begin{equation*}
    \mathfrak{m} = (\mathfrak{m}_\beta \oplus \mathfrak{m}_{\alpha+\beta} \oplus \mathfrak{m}_{2\alpha+\beta} \oplus \mathfrak{m}_{3\alpha+\beta}) \oplus \mathfrak{m}_{3\alpha+2\beta}.
\end{equation*}
This has $2$ almost-complex structures, $(+,+)$ and $(+,-)$. Here we are using the same notation as in the $\mathbb{C}P^{2n-1}$ example, where $(+,-)$ means that the sign $\epsilon_\gamma$ for all the roots that appear in the first summand, $\beta, \alpha+\beta , 2\alpha+\beta$ and $3\alpha+ \beta$, are $1$ and the sign $\epsilon_{3\alpha+ 2 \beta}=-1$.\\
The almost-complex structure $(+,+)$ is integrable.\\
For the non-integrable structure $(+,-)$, we see that it does not accept a Griffiths semi-positive metric by applying Lemma \ref{lemma} on the pair of positive roots $\gamma=3\alpha+ \beta$ and $\delta=-3\alpha - 2\beta$, since $\gamma+\delta=-\beta \in R_M$, but $\gamma-\delta=6\alpha +3\beta$ is not a root.\\

Lastly, we analyze the flag $G_2/SU(2) \times U(1)$ where $SU(2) \subset G_2$ represents the short root, the isotropy representation is
    \begin{equation*}
    \mathfrak{m} = (\mathfrak{m}_\alpha \oplus \mathfrak{m}_{\alpha+\beta}) \oplus \mathfrak{m}_{2\alpha+\beta} \oplus (\mathfrak{m}_{3\alpha+\beta} \oplus \mathfrak{m}_{3\alpha+2\beta}).
\end{equation*}
It has $4$ almost-complex structures $(+,+,+) , (+,+,-) , (+,-,+)$ and $(+,-,-)$. The first one is integrable. For the other three cases, we will exhibit a pair of roots, $\gamma$ and $\delta$, on $R_{M}^{+}$ such that $\gamma+\delta$ is a root, but $\gamma-\delta$ is not a root, and by Lemma \ref{lemma} it does not accept a quasi-Kähler metric of Griffiths semi-positive curvature.

\begin{enumerate}
    \item[1)] For the almost-complex structure $(+,+,-)$ take $\gamma=2\alpha+\beta$ and\\ $\delta=-3\alpha-\beta$.
    \item[2)] For the almost-complex structure $(+,-,+)$ take $\gamma=-2\alpha-\beta$ and\\ $\delta=3\alpha+\beta$.
      \item[3)] For the almost-complex structure $(+,-,-)$ take $\gamma=\alpha$ and $\delta=-3\alpha-\beta$.
\end{enumerate}
This concludes the proof.
\end{proof}

We believe that Lemma \ref{lemma} is too strong to allow quasi-Kähler flag manifolds with positive curvature except the flag manifolds with one isotropic summand, which are the simplest and the unique invariant metric has dual-Nakano semi-positive curvature.

\begin{conjecture}
If a flag manifold admits an almost-complex structure $J$ and a quasi-Kähler metric $\lambda$ of Griffiths semi-positive curvature then, it is isometrically biholomorphic to
\begin{equation*}
   (M_1 , g_1 ) \times ... \times (M_p , g_p ),
\end{equation*}
where $M_i$ are the Hermitian Symmetric spaces and $g_i$ its unique invariant metric.
\end{conjecture}

The example of $G_2$ was easy to compute because there are only $3$ flag manifolds associated with it, but $E_8$ has $40$ non-isomorphic flag manifolds and each of them has a large number of almost-complex structures, making the proof of the conjecture by direct computations, although possible, far too cumbersome. But, to endorse our beliefs on this conjecture, we will analyze a few flag manifolds associated to $F4$.

\begin{example}
    The roots space of $F_4$ has $4$ simple roots, which we denote by $\alpha, \beta, \gamma, \delta$. In terms of these simple roots the root space is
    \begin{equation}
    \begin{aligned}
        &R=\pm\{\alpha,\beta,\gamma,\delta,\\
        &\alpha+\beta,\beta+\gamma, \gamma +\delta ,\\
        &\alpha+\beta+ \gamma, \beta+\gamma +\delta , 2\beta+\gamma, \\
        &\alpha+\beta+2\gamma , \alpha+\beta+\gamma+\delta , \beta+2\gamma+\delta,\\
        &\alpha+2\beta+2\gamma, \alpha+\beta+2\gamma+\delta, \beta+2 \gamma+2 \delta,\\
        &\alpha+2\beta+2\gamma+\delta , \alpha+\beta+2\gamma+2\delta,\\
        &\alpha+2\beta +3\gamma +\delta, \alpha+2\beta +2\gamma+2\delta,\\
        &\alpha+2\beta +3\gamma+2\delta,\\
        &\alpha+2\beta +4\gamma+2\delta,\\
        &\alpha+3\beta +4\gamma+2\delta,\\
          &2\alpha+3\beta +4\gamma+2\delta\}.
         \end{aligned}
    \end{equation}

The maximal flag manifold $F_4 / T^4$ does not admit any quasi-Kähler metric of Griffiths semi-positive curvature for any almost-complex structure by Theorem \ref{Maximal}.

If we take the flag associated with the choice $\Pi_K = \{\beta,\gamma,\delta\}$, the resulting flag has $2$ summands
\begin{equation}
    \mathfrak{m}= \bigoplus_{\xi_\alpha =1}\mathfrak{m}_\xi \bigoplus _{\xi_\alpha =2}\mathfrak{m}_\xi
\end{equation}
where, by $\xi_\alpha = k$ we mean that when expressing the root $\xi$ in terms of the simple roots, $\alpha$ appears $k$ times.\\
Therefore, this flag has $2$ non-equivalent almost-complex structure, $(+,+)$ and $(+,-)$, the first is integrable and $F_4$ does not have any Kähler flag manifold of Griffiths semi-positive curvature. The second is non-integrable, and we only need to apply Lemma \ref{lemma} to the pair $\alpha$ and $-2\alpha - 3\beta - 4\gamma-2\delta$.\\

We will analyze one last flag manifold for $F_4$, for the choice $\pi_K = \{\alpha,\beta,\delta\}$. This flag has $4$ summands
\begin{equation}
    \mathfrak{m}= \bigoplus_{\xi_\alpha =1}\mathfrak{m}_\xi \bigoplus _{\xi_\alpha =2}\mathfrak{m}_\xi \bigoplus_{\xi_\alpha =3}\mathfrak{m}_\xi \bigoplus _{\xi_\alpha =4}\mathfrak{m}_\xi
\end{equation}
and $8$ non-equivalent almost-complex structures, one is integrable $(+,+,+,+)$. To prove that the other structures do not accept quasi-Kähler metrics of Griffiths semi-positive curvature, we exhibit for each a pair of roots on $R_{M}^{+}$ that the sum is root, but the difference is not a root, and use Lemma \ref{lemma} to conclude our claim.

\begin{enumerate}
    \item [1)] For $J=(+,+,+,-),(+,+,-,+),(+,+,-,-), $ take $\gamma$ and $\beta+\gamma+\delta$
    \item[2)] For $J= (+,-,+,+) , (+,-,+,-)$ take $\gamma$ and $\alpha+2\beta+2\gamma+\delta$
    \item[3)] For $J=(+,-,-,+)$ take $\alpha+2\beta+4\gamma+2\delta$ and $-\alpha-2\beta-3\gamma-2\delta$
    \item[4)] For $J=(+,-,-,-)$ take $\gamma$ and $-\alpha-2\beta-4\gamma-2\delta$.
\end{enumerate}

\end{example}

\section{Curvature positivity of the complex projective space}

We have shown that, the only invariant Kähler metric on the complex projective space with Griffiths semi-positive curvature is the Fubini-Study metric. In this section, we want to characterize all invariant almost-Hermitian metrics on the complex projective with Griffiths semi-positive curvature.\\

First, the only $SU(n)$ invariant metric on $\mathbb{C}P^{n-1} = SU(n)/S( U(n-1) \times U(1)$ is the Fubini-study metric, which we have seen to have semi-positive curvature. Therefore, we will focus on metrics invariant by $Sp(n)$ on $\mathbb{C}P^{2n-1} =Sp(n)/Sp(n-1)\times U(1)$.\\

As we have seen, the flag manifold $Sp(n)/Sp(n-1)\times U(1)$  has $2$ summands therefore, there are $2$ almost-complex structures, the integrable $J_1 =(+,+)$ and the non-integrable $J_2 =(+,-)$, and the invariant metrics can be parametrized by $\lambda_t =(1,t)$.

We have already seen that $ (\frac{Sp(n)}{Sp(n-1)\times U(1)}, J_1 ,\lambda_t )$  is dual-Nakano positive for $t=2$ because it is the Fubini-Study metric, and by Proposition \ref{lambda=1} we know that $t=1$ is dual-Nakano semi-positive. Our next results complete the characterization of invariant metrics on $\mathbb{C}P^{n}$ with Griffiths semi-positive curvature.

\begin{proposition}
     The projective space $\mathbb{C}P^{2n-1} =Sp(n)/Sp(n-1)\times U(1)$ endowed with the non-integrable almost-complex structure $(+,-)$ does not admit invariant metric with Griffiths semi-positive curvature.
\end{proposition}
\begin{proof}
    Simply note that every invariant metric $(1,t)$ on this almost-complex structure is quasi-Kähler, then by Theorem \ref{Height3} it does not have  Griffiths semi-positive curvature.
\end{proof}

\begin{theorem}
The projective space $\mathbb{C}P^{2n-1} =Sp(n)/Sp(n-1)\times U(1)$ endowed with the integrable almost-complex structure $(+,+)$ and the invariant metric $(1,t)$ is Griffiths and dual-Nakano semi-positive if, and only if, $t\geq 1$. Moreover, $\mathbb{C}P^{2n-1}$ is Griffiths and dual-Nakano semi-positive if, and only if, $t> 1$.
\end{theorem}
\begin{proof}
Using the notation $R_{\alpha,\overline{\beta}, \gamma,\overline{\delta}} = \lambda (R(X_\alpha , X_{-\beta})X_\gamma , X_{-\delta})$, we compute all the non-zero curvature terms on $Sp(n)/Sp(n-1)\times U(1)$.

\begin{equation}\label{curvatura 1}
    \begin{aligned}
& R_{2\lambda_1 ,\overline{2\lambda_1},\lambda_1 - \lambda_j ,\overline{\lambda_1 - \lambda_j}} =   m_{2\lambda_1 , -\lambda_1+\lambda_j}^2        \\
& R_{2\lambda_1 ,\overline{2\lambda_1},\lambda_1 + \lambda_j ,\overline{\lambda_1 + \lambda_j}} =    m_{2\lambda_1 , -\lambda_1-\lambda_j}^2     \\ 
&R_{\lambda_1 -\lambda_j , \overline{\lambda_1 - \lambda_j}, 2\lambda_1 ,\overline{2\lambda_1}}=  m_{\lambda_1 -\lambda_j , -2\lambda_1}^2 (t-1)\\
&R_{\lambda_1 -\lambda_j , \overline{\lambda_1 - \lambda_j}, \lambda_1 - \lambda_k ,\overline{\lambda_1 - \lambda_k}}= m_{\lambda_1 - \lambda_j , -\lambda_1 +\lambda_k}^2 \\
&R_{\lambda_1 -\lambda_j , \overline{\lambda_1 - \lambda_j},\lambda_1 +\lambda_j ,\overline{\lambda_1 +\lambda_j}} = m_{\lambda_1 - \lambda_j , -\lambda_1 - \lambda_j}^2 + m_{\lambda_1 - \lambda_j , \lambda_1 +\lambda_j}^2 (\frac{1}{t}-1)\\
&R_{\lambda_1 -\lambda_j , \overline{\lambda_1 - \lambda_j},\lambda_1 +\lambda_k ,\overline{\lambda_1+\lambda_k}} = m_{\lambda_1 - \lambda_j , -\lambda_1 - \lambda_k}^2\\
&R_{\lambda_1 +\lambda_j , \overline{\lambda_1 + \lambda_j}, 2\lambda_1 ,\overline{2\lambda_1}}= m_{\lambda_1 +\lambda_j , -2\lambda_1}^2 (t-1) \\
&R_{\lambda_1 +\lambda_j , \overline{\lambda_1 + \lambda_j}, \lambda_1 + \lambda_k ,\overline{\lambda_1 + \lambda_k}}=m_{\lambda_1 +\lambda_j , -\lambda_1 - \lambda_k}^2 \\
&R_{\lambda_1 +\lambda_j , \overline{\lambda_1 + \lambda_j},\lambda_1 -\lambda_j ,\overline{\lambda_1 -\lambda_j}} =m_{\lambda_1 + \lambda_j , -\lambda_1 + \lambda_j}^2 + m_{\lambda_1 + \lambda_j , \lambda_1 -\lambda_j}^2 (\frac{1}{t}-1)\\
&R_{\lambda_1 +\lambda_j , \overline{\lambda_1 + \lambda_j},\lambda_1 -\lambda_k ,\overline{\lambda_1-\lambda_k}} = m_{\lambda_1 +\lambda_j , -\lambda_1 +\lambda_k}^2 \\
&R_{2\lambda_1 , \overline{2\lambda_1},2\lambda_1 , \overline{2\lambda_1}} =t B(H_{2\lambda_1} , H_{2\lambda_1} ) = \frac{t}{n+1}\\
&R_{\lambda_1 - \lambda_j , \overline{\lambda_1 - \lambda_j},\lambda_1 - \lambda_j , \overline{\lambda_1 - \lambda_j}} = B(H_{\lambda_1 - \lambda_j} , H_{\lambda_1 - \lambda_j} ) = \frac{1}{2(n+1)}\\
&R_{\lambda_1  +\lambda_j , \overline{\lambda_1 + \lambda_j},\lambda_1 + \lambda_j , \overline{\lambda_1 + \lambda_j}} = B(H_{\lambda_1 + \lambda_j} , H_{\lambda_1 + \lambda_j} ) = \frac{1}{2(n+1)}
    \end{aligned}
\end{equation}
From this, we see that if $t<1$ some curvature terms are negative, and therefore not Griffiths or dual-Nakano semi-positive. Also for $t=1$ it is not Griffiths or dual-Nakano positive. Now, we compute the rest of the curvature tensors that are not null.
\begin{equation}\label{curvatura2}
    \begin{aligned}
        & R_{\lambda_1 - \lambda_j ,\overline{2\lambda_1},2\lambda_1, \overline{\lambda_1 - \lambda_j}} =   m_{2\lambda_1 , -\lambda_1 +\lambda_j}^2   =\frac{1}{2(n+1)} \\
        & R_{\lambda_1 + \lambda_j ,\overline{2\lambda_1},2\lambda_1, \overline{\lambda_1 + \lambda_j}} =   m_{2\lambda_1 , -\lambda_1-\lambda_j}^2   =\frac{1}{2(n+1)} \\
        & R_{2\lambda_1, \overline{\lambda_1 - \lambda_j},\lambda_1 -\lambda_j \overline{2\lambda_1}}=   m_{\lambda_1 - \lambda_j , -2\lambda_1}^2   =\frac{1}{2(n+1)} \\
&R_{\lambda_1 -\lambda_k, \overline{\lambda_1 - \lambda_j},\lambda_1 -\lambda_j \overline{\lambda_1 - \lambda_k}}= m_{\lambda_1 -\lambda_j , -\lambda_1 +\lambda_k}^2  =\frac{1}{4(n+1)} \\
&R_{\lambda_1 +\lambda_j, \overline{\lambda_1 - \lambda_j},\lambda_1 -\lambda_j \overline{\lambda_1 +\lambda_j}}= -m_{\lambda_1 - \lambda_j , \lambda_1 +\lambda_j}^2\frac{1}{t} + m_{\lambda_1 - \lambda_j , -\lambda_1 - \lambda_j}^2   =\frac{1}{2(n+1)}(1-\frac{1}{t})  \\
&R_{\lambda_1 +\lambda_k, \overline{\lambda_1 - \lambda_j},\lambda_1 -\lambda_j \overline{\lambda_1 +\lambda_k}}= m_{\lambda_1 - \lambda_j , -\lambda_1 -\lambda_k}^2  =\frac{1}{4(n+1)} \\
& R_{2\lambda_1, \overline{\lambda_1 + \lambda_j},\lambda_1 +\lambda_j \overline{2\lambda_1}}=   m_{\lambda_1 +\lambda_j , 2\lambda_1}^2   =\frac{1}{2(n+1)}  \\
&R_{\lambda_1 +\lambda_k, \overline{\lambda_1 + \lambda_j},\lambda_1 +\lambda_j \overline{\lambda_1 + \lambda_k}}= m_{\lambda_1 +\lambda_j ,-\lambda_1 - \lambda_k }^2     =\frac{1}{4(n+1)} \\
&R_{\lambda_1 -\lambda_j, \overline{\lambda_1 + \lambda_j},\lambda_1 +\lambda_j \overline{\lambda_1 -\lambda_j}}=  -\frac{1}{t}m_{\lambda_1 +\lambda_j , \lambda_1 - \lambda_j}^2 + m_{\lambda_1 +\lambda_j , -\lambda_1 +\lambda_j}^2   =\frac{1}{2(n+1)}(1-\frac{1}{t}) \\
&R_{\lambda_1 -\lambda_k, \overline{\lambda_1 + \lambda_j},\lambda_1 +\lambda_j \overline{\lambda_1 -\lambda_k}}=  m_{\lambda_1 +\lambda_j , -\lambda_1 - \lambda_k}^2    =\frac{1}{4(n+1)}  .\\
    \end{aligned}
\end{equation}

 We want to prove that the metric $(1,t)$ is dual-Nakano positive if, and only if, $ t >1$, we already know it is positive for $t=2$ because it is the Fubini-Study metric.

Given $u = \sum u^{\alpha,\beta} X_\alpha \otimes X_\beta $ we want to analyze if
\begin{equation}
    \sum_{\alpha,\beta,\gamma,\delta} u^{\alpha,\delta} \overline{u^{\beta,\gamma}} R_{\alpha,\overline{\beta},\gamma,\overline{\delta}} \geq 0.
\end{equation}
Excluding the terms that are zero, we are left with
\begin{equation}
      \sum_{\alpha,\gamma} u^{\alpha,\gamma} \overline{u^{\alpha,\gamma}} R_{\alpha,\overline{\alpha},\gamma, \overline{\gamma}}  + u^{\alpha,\alpha} \overline{u^{\gamma,\gamma}} R_{\alpha,\overline{\gamma},\gamma, \overline{\alpha}}.
\end{equation}
Note that the mixed terms $u^{\alpha,\gamma} X_\alpha \otimes X_\gamma$ always contribute positively to the sum if $t\geq 1$, because $R_{\alpha,\overline{\alpha},\gamma}^\gamma$ is always positive. Thus, we may assume that\\ $u=\sum u^{\alpha,\alpha}X_\alpha \otimes X_\alpha$. The sum becomes
\begin{equation}
       \sum_{\alpha,\gamma}    u^{\alpha,\alpha} \overline{u^{\gamma,\gamma}} R_{\alpha,\overline{\gamma},\gamma,\overline{\alpha}} .
\end{equation}

Fix an ordering on $R_M$,\\ $(\alpha_1 ,..., \alpha_{2n+1}) = (\lambda_1 -\lambda_2 , \lambda_1 +\lambda_2 ,..., \lambda_1 -\lambda_n , \lambda_1 +\lambda_n , 2\lambda_1)$ and define a matrix $R^t = [R^{t}_{\alpha_i \overline{\alpha_j}\alpha_j , \overline{\alpha_i}}]$. Then, identifying the $u = \sum u^{\alpha,\alpha} X_\alpha \otimes X_\alpha \in T^{1,0}M \otimes T^{1,0}M$ with $u= (u^{\alpha_1 \alpha_1} ,..., u^{\alpha_{2n+1}\alpha_{2n+1}}) \in \mathbb{C}^{2n+1}$ we have:

\begin{equation}
       \sum_{\alpha,\gamma}    u^{\alpha,\alpha} \overline{u^{\gamma,\gamma}} R_{\alpha,\overline{\gamma},\gamma,\overline{\alpha}}  = [u]^{T} [R^t ] \overline{[u]}.
\end{equation}
The positivity of this for every $u \in \mathbb{C}^{2n+1}$ is equivalent to the eigenvalues of $[R^t]$ being positive. Computing the structure constant on Equation \ref{curvatura 1} and \ref{curvatura2}, the matrix $[R^t]$ has the below form:

\begin{equation}
   \frac{1}{2(n+1)} \left(\begin{array}{cccccccc}
       1   &1-\frac{1}{t} & \frac{1}{2} &\frac{1}{2} &\hdots &\frac{1}{2} &\frac{1}{2} &1\\
       1-\frac{1}{t} & 1 & \frac{1}{2}& \frac{1}{2}  &\hdots &\frac{1}{2} &\frac{1}{2} &1\\
     \frac{1}{2} &\frac{1}{2}  &1   &1-\frac{1}{t}  &\hdots &\frac{1}{2} &\frac{1}{2} &1\\
     \frac{1}{2}& \frac{1}{2}     &1-\frac{1}{t} & 1   &\hdots &\frac{1}{2} &\frac{1}{2} &1\\
      \vdots & \vdots  & \vdots & \vdots &\ddots &  &  &\vdots\\
     \frac{1}{2} &\frac{1}{2} &\frac{1}{2} &\frac{1}{2} &\hdots  &1   &1-\frac{1}{t} &1\\
     \frac{1}{2} &\frac{1}{2} &\frac{1}{2} &\frac{1}{2} &\hdots     &1-\frac{1}{t} &1 &1\\
      1 &1 &1 &1 &\hdots &1&1  &2t
    \end{array}\right).
\end{equation}

Since the function $t \mapsto [R^t ]$ is analytic and $[R^t]$ is symmetric, we can parametrize the eigenvalues $\lambda_i (t)$ continuously \cite[Theorem 3.9.1]{tyrtyshnikov1997brief}.

As for $t=2$ all eigenvalues are positive, if we prove that $0$ is not an eigenvalue of $[R^t ]$ for $t>1$, then, due to the continuity of the eigenvalues, all eigenvalues of $[R^t ]$ are positive for $t>1$. Also note that $0$ is not an eigenvalue of a matrix if, and only if, the matrix is invertible. Thus, it suffices to prove that $[R^t ]$ is invertible for $t>1$, we do so by proving that its columns form a basis for $\mathbb{C}^{2n-1}$.\\

Let $v_1 ,..., v_{2n-1}$ be the columns of $[R^t ]$, given $a_1 ,..., a_{2n-1} \in \mathbb{C}$ such that
\begin{equation*}
    a_1 v_1 +...+a_{2n-1}v_{2n-1} =0 ,
\end{equation*}
we will show that $a_1,..., a_{2n-1}=0$. Analyzing the row $2i$ and $2i-1$ for $i=1,...,n-1$ we have
\begin{equation*}
    \begin{aligned}
       & \frac{1}{2}(a_1 +... a_{2n-2} ) +a_{2n-1} +\frac{1}{2}(a_{2i}+a_{2i-1} ) -\frac{a_{2i}}{t}=0\\
        & \frac{1}{2}(a_1 +... a_{2n-2} ) +a_{2n-1} +\frac{1}{2}(a_{2i}+a_{2i-1} ) -\frac{a_{2i-1}}{t}=0 .
    \end{aligned}
\end{equation*}
Subtracting one with another, we get $a_{2i-1}= a_{2i}$  and therefore, for $i=1 ,..., 2n-2$ we have
 \begin{equation}\label{LI}
     \frac{1}{2}(a_1 +...+ a_{2n-2}) +a_{2n-1} +a_i -\frac{a_i }{t}=0 .
 \end{equation}
Subtracting this equation for two different indexes, $i \neq j$, we get
\begin{equation*}
    (a_i - a_j )(1 - \frac{1}{t}) =0,
\end{equation*}
since $t \neq 1$, we have that $a_i = a_j$ for every $i,j = 1,..., 2n-2$.\\ 

Finally, we analyze the last row, which gives us
\begin{equation*}
    a_1 +...+a_{2n-2} +2t a_{2n-1}=2t a_{2i-1} =0 .
\end{equation*}
Using the fact that $a_1 =...=a_{2n-2}$ and combining this Equation with Equation \ref{LI} we get that $a_1 =...=a_{2n-1}=0$, which concludes the proof.

\end{proof}

{\large\textbf{Acknowledgements.}} The authors would like to thank professor Lino Grama for his helpful suggestions and encouragement. G. Galindo is supported by the Coordination for the Improvement of Higher Education Personnel CAPES.

\bibliographystyle{alpha}
\bibliography{ref.bib}

\end{document}